%% file: NLocalC.tex
\documentclass[a4paper,12pt]{amsart}
\input{Definitions.tex}

\begin{document}

\title{A colocalization spectral sequence}
\author{S. Shamir}
\address{Department of Mathematics, University of Bergen, 5008 Bergen, Norway}
\email{Shoham.Shamir@math.uib.no}
\date{\today}

\begin{abstract}
Colocalization is a right adjoint to the inclusion of a subcategory. Given an $\sphere$-algebra $R$, one would like a spectral sequence which relates colocalization in the homotopy category of $R$-modules to an appropriate colocalization in the derived category of $\pi_*R$-modules. We show that, under suitable conditions, such a spectral sequence exists. This generalizes Greenlees' local-cohomology spectral sequence. The colocalization spectral sequence introduced here is associated with a localization spectral sequence, which is shown to be universal in an appropriate sense. We apply the colocalization spectral sequence to the cochains of certain loop spaces, yielding a non-commutative local-cohomology spectral sequence converging to the shifted cohomology of the loop space, a result dual to the local-cohomology theorem of~\cite{DwyerGreenleesIyengar}. An application to the abutment term of the Eilenberg-Moore spectral sequence is also presented.
\end{abstract}


\maketitle


\section{Introduction}
\label{sec: Introduction}

Let $\Derived$ be a category and let $\cderived$ be a subcategory of $\Derived$. A \emph{$\cderived$-colocalization} of $X\in \Derived$ is a morphism $\eta:C \to X$ in $\Derived$ with $C\in \cderived$ and such that $\hom_\Derived(T,\eta)$ is an isomorphism for every $T\in \cderived$. For historical reasons we will call $C$ a \emph{$\cderived$-cellular approximation} of $X$ and denote it by $\cell_\cderived^\Derived X$.

Now consider the following setup. Let $R$ be an $\sphere$-algebra (or a dga). We denote by $\pi_*R$ the graded algebra of stable homotopy groups of $R$, if $R$ is a dga then these are understood to be the homology groups of $R$. Let $\tc$ be a class of graded left $\pi_*R$-modules which is closed under submodules, quotient modules, coproducts, extensions and degree shifts. Such a class is called a \emph{hereditary torsion class} (see Definition~\ref{def: Hereditary torsion theory}). Let $\Derived$ be the derived category of left $R$-modules, also called the homotopy category of left $R$-modules, and let $\cderived$ be the subcategory of $R$-modules $M$ such that $\pi_*(M)\in \tc$. Denote by $\Derived_*$ the derived category of bounded above $\pi_*R$-complexes and let $\cderived_*$ be the subcategory of $\Derived_*$ consisting of all complexes $\Xcomplex$ such that $H^n(\Xcomplex) \in \tc$ for all $n$. We shall use the notation $\cell^R_\tc$ for $\cell_\cderived^\Derived$ and use $\cell_{\tc}^{\pi_*R}$ for $\cell_{\cderived_*}^{\Derived_*}$. A \emph{$\tc$-colocalization spectral sequence} for an $R$-module $M$ is a spectral sequence of the form:
\begin{equation*}
    \Espec^2_{p,q}= H^{-p}(\cell_{\tc}^{\pi_*R} \pi_*M)_q \ \Rightarrow \ \pi_{p+q}(\cell_\tc^R M)
\end{equation*}
with conditional convergence.

The Greenlees spectral sequence defined in~\cite{GreenleesCommutativeAlgebraInGroupCohomology} (see also~\cite{BensonGreenleesCommAlgLieGrps}) can be viewed as a colocalization spectral sequence, as was shown by Dwyer, Greenlees and Iyengar in~\cite{DwyerGreenleesIyengar}. The setup for this spectral sequence is the following. Let $G$ be a compact Lie group, for simplicity assume that $G$ is connected, and let $k$ be a field. Denote by $R$ the dga $\chains^*(\B G;k)$, i.e. singular cochains on the classifying space of $G$. Then $\pi_*(R)$ is the graded-commutative ring $H^*(\B G;k)$. Let $\ideal$ be the augmentation ideal $\ker(H^*(\B G;k) \to k)$. An $H^*(\B G;k)$-module $\module$ is \emph{$\ideal$-power torsion} if for every $m\in \module$, $\ideal^n m=0$ for some $n$. Denote the class of $\ideal$-power torsion modules by $\tc$. The Greenlees spectral sequence can be written as
\[ (*) \qquad \Espec^2_{p,q}= H^{-p}(\cell_{\tc}^{H^*(\B G;k)} H^*(\B G;k))_q \ \Rightarrow \ \pi_{p+q}(\cell_\tc^R R)\]
Using~\cite{DwyerGreenleesIyengar} and~\cite{DwyerGreenlees} it is easy to see how the usual form for the Greenlees spectral sequence:
\[(**) \qquad \Espec^2_{p,q}= H_\ideal^{-p}(H^*(\B G;k))_q \ \Rightarrow \ H_{p+q-d}(\B G;k)\]
follows from $(*)$. One of the applications for the spectral sequence presented in this paper is an analogous result to $(**)$ for certain loop-spaces, see Theorem~\ref{the: Application to the chains of loops of elliptic}.

Similar to $(*)$ above is a colocalization spectral sequence presented by Greenlees and May in~\cite{GreenleesMay}. In that spectral sequence, $R$ is a commutative $\sphere$-algebra, $\ideal$ is a finitely generated ideal of $\pi_*R$ and $\tc$ is the class of $\ideal$-power torsion $\pi_*R$-modules. Another instance of a spectral sequence of similar flavour is~\cite{HoveyStricklandLocalCohoBP}, though the setting there is quite different than the one considered here.

The existing methods of constructing such colocalization spectral sequences may fail in non-commutative situations. For example, the construction of the Greenlees spectral sequence requires $R$ to be a commutative $\sphere$-algebra. While the methods of~\cite{HoveyStricklandLocalCohoBP} do not necessitate the use of a commutative $\sphere$-algebra, these may also fail in non-commutative situations, as we now explain. The method of~\cite{HoveyStricklandLocalCohoBP} can be roughly described as applying a certain localization functor to a filtration which is the lift of an injective resolution (the relevance of localization functors will be explained further on). For this method to succeed, it requires that the corresponding localization functor on $\pi_*(R)$-complexes be exact on injective $\pi_*(R)$-modules. But this need not be true in non-commutative situations; see~\cite[Example 5.2]{ShamirEMSS} for a simple example of such a case. The aim of this paper is to rectify these shortcomings by constructing a colocalization spectral sequence for cases where $R$ is a non-commutative $\sphere$-algebra. The actual construction is carried out in Section~\ref{sec: Construction of the spectral sequences}, while the convergence issues are dealt with in the main theorems below. As we shall see, the spectral sequence constructed here in fact generalizes the existing ones. It should also be noted that many of the results of this paper hold in any triangulated category with a compact generator, see Remark~\ref{rem: Generality of the spectral sequence}.


\begin{specialthm}
\label{thm: First Theorem}
Let $R$ be an $\sphere$-algebra and let $\tc$ be a hereditary torsion class of $\pi_*R$-modules. For any $R$-module $M$ there is a spectral sequence $\Espec^r_{p,q}$ whose $\Espec^2$-term is
\[\Espec^2_{p,q}= H^{-p}(\cell_{\tc}^{\pi_*R} \pi_*M)_q\]
If the $\pi_*R$-chain complex $\cell_\tc^{\pi_*R} \pi_*M$ is bounded below (i.e. has zero homology below a certain degree), then the spectral sequence converges strongly to $\pi_{p+q}(\cell_\tc^R M)$.
\end{specialthm}

Of course, for Theorem~\ref{thm: First Theorem} to be useful, one must find cases for which the condition on the spectral sequence is satisfied. Let $R$ be an $\sphere$-algebra and let $\ideal$ be a two-sided ideal of $\pi_*R$. We say $\ideal$ is \emph{almost-commutative} if $\ideal=(x_1,...,x_n)$, where $x_1,..,x_n$ are homogeneous elements,  and the image of $x_i$ in $\pi_*R/(x_1,...,x_{i-1})$ is central. The next result shows that very mild conditions are needed for a colocalization spectral sequence to converge strongly.

\begin{specialthm}
\label{thm: Second Theorem}
Let $R$ be an $\sphere$-algebra and let $\ideal$ be an almost-commutative ideal of $\pi_*R$. Suppose that $\pi_*R$ is Noetherian or that $\pi_*R$ is graded-commutative. Let $\tc$ be the class of $\ideal$-power torsion modules. Then for any $R$-module $M$ there exists a strongly convergent $\tc$-colocalization spectral sequence:
\[ \Espec^2_{p,q}= H^{-p}(\cell_{\tc}^{R_*} \pi_*(M))_q \ \Rightarrow \ \pi_{p+q}(\cell_\tc^R M)\]
\end{specialthm}

The next result, Theorem~\ref{thm: Third Theorem}, is similar in flavour to Theorem~\ref{thm: First Theorem}. It should be noted that various other conditions under which Theorem~\ref{thm: Third Theorem} holds can be easily produced, and so it serves as an example of the versatility of the theory. Recall that for a ring $k$ the notation $Hk$ stands for the corresponding Eilenberg-Mac Lane $\sphere$-algebra. For an $R$-module $M$ we use $\loc{M}$ to denote the localizing subcategory generated by $M$.

\begin{specialthm}
\label{thm: Third Theorem}
Let $k$ be a field and let $R$ be a connective $Hk$-algebra with augmentation $R \to Hk$. Suppose that $\pi_*R$ is left Noetherian. Let $\ideal$ be the kernel of the surjection $\pi_*R \to k$ induced by the augmentation and let $\tc$ be the hereditary torsion class of $\ideal$-power torsion modules. If $M$ is a bounded-above $R$-module, then there is a strongly convergent $\tc$-colocalization spectral sequence for $M$:
\[ \Espec^2_{p,q}= H^{-p}(\cell_{\tc}^{R_*} \pi_*(M))_q \ \Rightarrow \ \pi_{p+q}(\cell_\tc^R M)\]
Moreover, $\cell_\tc^R M\simeq \cell_\loc{Hk}^R M$ and $\cell_\tc^{\pi_*R}\pi_*M \simeq \cell_\loc{k}^{\pi_*R}\pi_*M$. Hence the spectral sequence can be written as:
\[ \Espec^2_{p,q}= H^{-p}(\cell_{\loc{k}}^{R_*} \pi_*(M))_q \ \Rightarrow \ \pi_{p+q}(\cell_\loc{Hk}^R M)\]
\end{specialthm}

Associated to every colocalization in a derived category is its counterpart, localization. In our setup, where $\tc$ is a hereditary torsion class of $\pi_*R$-modules, this localization functor comes under the guise of \emph{$\tc$-nullification}. An $R$-module is $\tc$-null if it admits no nontrivial morphisms from any $\tc$-torsion module. A \emph{$\tc$-nullification} of an $R$-module $M$ is a morphism $M \to N$ which is initial among morphisms from $M$ to $\tc$-null modules, see Definition~\ref{def: Nullification}. It is now clear what a $\tc$-nullification spectral sequence should mean.

The colocalization spectral sequence constructed in this paper is in fact induced from a particular nullification spectral sequence constructed in Lemma~\ref{lem: Spectral sequence for nullification}. Hence, to each of the spectral sequences in Theorems~\ref{thm: First Theorem}, \ref{thm: Second Theorem} and~\ref{thm: Third Theorem} there is an associated nullification spectral sequence with similar convergence properties. One reason why the colocalization spectral sequence is relegated a more prominent role is that the convergence conditions are more visible there. However, the nullification spectral sequence deserves to be mentioned for its unique property of being universal in an appropriate sense. Its universality means that any \emph{proper $\tc$-nullification spectral sequence} (see Definition~\ref{def: Proper nullification spectral sequence}) is isomorphic to it. This is the essence of our fourth main result, Theorem~\ref{thm: Universality for nullification}.

One source of examples for this theory is the chains of loop spaces. Let $k$ be a commutative ring and consider the dga $R=\chains_*(\kangr X;k)$ -- the singular chains on the Kan loop group of a connected space $X$. Under certain conditions Theorem~\ref{thm: Second Theorem} can be applied to the derived category of $R$, thereby achieving a result analogous to the Greenlees spectral sequence described in $(**)$. This will be done in Theorem~\ref{the: Application to the chains of loops of elliptic}.

Using a similar set-up, we can sometimes apply Theorem~\ref{thm: Third Theorem} to measure the difference between the abutment term of the Eilenberg-Moore spectral sequence and the cohomology of the fibre in terms of a nullification spectral sequence. This is done in Theorem~\ref{the: Another formulation of Application to rational EMSS}.

Another application of the colocalization spectral sequence presented here is in a recent preprint by J{\o}rgensen~\cite{Jorgensen}, where it is used to develop a duality between left and right modules over a dga.

\subsection*{Organization of this paper}

Section~\ref{sec: Modules over a graded ring} describes the construction of the derived category of a graded ring. The section contains no new material and can be ignored by anyone familiar with this subject.

We then present some background material on cellular approximations in Section~\ref{sec: Cellular approximation and nullification}. All of this material is well known in one form or another. Next, in Sections~\ref{sec: Colocalization with respect to a hereditary torsion theory} and~\ref{sec: Explicit construction of tc-nullification}, we consider colocaliztion over a graded ring $\ring$ with respect to a hereditary torsion class on $\ring$-modules and recall material from~\cite{ShamirTorsion}. The main result of~\cite{ShamirTorsion} is recalled in Theorem~\ref{the: Nullification as double endomorphism}. It is this theorem that enables the construction of a colocalization spectral sequence.

Once we have all the necessary background material we can construct the spectral sequences. This is done in Section~\ref{sec: Construction of the spectral sequences}. As noted above, the construction of a spectral sequence with the desired $\Espec^2$-page and showing the spectral sequence converges are separate matters. Thus, the proofs of Theorems~\ref{thm: First Theorem}, \ref{thm: Second Theorem} and~\ref{thm: Third Theorem} are given in the next section, Section~\ref{sec: Proofs}. Following that, the universality property of the nullification spectral sequence and a comparison to the Greenlees spectral sequence are given in Section~\ref{sec: Universality of the Nullfication Spectral Sequence}.

After proving the main theorems we give the two applications of the colocalization spectral sequence. The first application, given in Section~\ref{sec: Application: the chains of loops on an elliptic space}, can be viewed as a non-commutative analog of the Greenlees spectral sequence $(**)$. 
The second application is presented in Section~\ref{sec: Application: the target of the Eilenberg-Moore spectral sequence} and is related to the abutment term of the Eilenberg-Moore spectral sequence. 

\subsection*{Setting, conventions and notation}

We work with $\sphere$-algebras in the sense of~\cite{EKMM}, where $\sphere$ stands for the sphere spectrum. Let $R$ be an $\sphere$-algebra, unless otherwise noted an $R$-module means a \emph{left} $R$-module. The homotopy category of $R$-modules will be called the \emph{derived category of $R$} and be denoted by $\Derived(R)$. The \emph{homotopy groups} of an $R$-module $M$ are its stable homotopy groups, denoted $\pi_* M$. An $R$-module $M$ is called \emph{bounded-above} if $\pi_i M=0$ for all $i>n$ for some $n$. An $\sphere$-algebra $R$ is called \emph{connective} if $\pi_i(R)=0$ for all $i<0$. Note that if $R$ is a dga and $M$ an dg-$R$-module then the homotopy groups of $M$ are the homology groups of $M$.

For $\sphere$-algebras and modules over $\sphere$-algebras we follow~\cite{DwyerGreenleesIyengar} in both notation and terminology. Thus, for an $\sphere$-algebra $R$ and for $R$-modules $M$ and $N$ the notation $M\otimes_R N$ stands for the smash product of $M$ and $N$ over $R$: $M \wedge_R N$. Similarly, $\Hom_R(M,N)$ stands for the function spectrum $F_R(M,N)$. As in~\cite{DwyerGreenleesIyengar}, both the function spectrum and the smash product are taken in the derived sense. This means that we always assume to have replaced our modules by appropriate resolutions before applying the relevant functor.

Let $X$ be a connected space and let $k$ be a commutative $\sphere$-algebra. We use $\kangr X$ to denote the Kan loop group of $X$, which is equivalent to the loop space of $X$. The cochains of $X$ with coefficients in $k$ is the commutative $\sphere$-algebra $\chains^*(X;k)=\Hom_\sphere(\Sigma^\infty X_+,k)$. The chains of $\kangr X$ with coefficients in $X$ is the $\sphere$-algebra $\chains_*(\kangr X;k)=k\otimes_\sphere \Sigma^\infty (\kangr X)_+$.

A \emph{graded ring} means a $\Int$-graded ring. A module over a graded ring is similarly $\Int$-graded. We shall reserve calligraphic font for graded rings their modules. Let $\ring$ be a graded ring. As above an $\ring$-module will mean a \emph{left} $\ring$-module. An \emph{$\ring$-complex} is a chain complex of left $\ring$-modules. Hence an $\ring$-complex is in fact bigraded. To avoid confusion we denote the grading of an $\ring$-module $\module$ by $\module\ing{i}$. If $m \in \module$ then we denote the degree of $m$ by $|m|$.

We say that $\ring$ is \emph{connective} if $\ring \ing{i}=0$ for all $i<0$. Ideals of $\ring$ are always graded (homogeneous) ideals. Given an $\ring$-module $\module$, we denote by $\shift^n\module$ the twisted grading shift of $\module$, see Definition~\ref{def: definition of the twisted grading shift functor}. We say an $\ring$-module $\module$ has \emph{bounded grading} if there is some $b$ such that $\module\ing{i}=0$ for all $i>b$. We shall call such $b$ a \emph{grading bound} of $\module$.

Since we follow the conventions of algebraic topology, we say a graded ring $\ring$ is \emph{graded-commutative} if $xy=(-1)^{|x|\cdot|y|}yx$ for all $x,y, \in \ring$. We say that an element $x\in\ring$ is \emph{central} if $x$ is central in the graded sense, i.e. $xy=(-1)^{\deg x \cdot \deg y} yx$ for all $y\in \ring$. Note that one could replace these with the usual notions of commutativity and centrality without affecting the results in this paper. 

It is taken for granted that every $\ring$-module is also a $\ring$-complex concentrated in degree zero. For $\ring$-complexes $\Xcomplex$ and $\Ycomplex$ we use $\Hom_\ring(\Xcomplex,\Ycomplex)$ for the usual chain complex of homomorphisms, albeit in this case the homomorphism groups are graded (see Definition~\ref{def: definition of graded Hom over a graded ring}). We will automatically assume $\Hom_\ring(-,-)$ is derived, as is the case for $\sphere$-algebras. We shall denote the suspension (translation) functor on $\Derived(\ring)$ by $\trans$. Hence, for an $\ring$-module $\module$ the notation $\trans\module$ stands for the $\ring$-complex with $\module$ in degree $1$ and zero elsewhere.

In a triangulated category $\Derived$ we use the notation $\hom_\Derived^*(-,-)$ for the graded abelian group of homomorphisms between two objects of $\Derived$. As always, this grading is induced the suspension functor $\Sigma$. We use the convention that when $\Derived$ is the derived category of a graded ring then $\hom_\Derived^*(-,-)$ is bigraded. The second grading comes of course from the twisted grading shift functor $\shift$. Thus $\hom_\Derived^{0,n}(\Xcomplex,\Ycomplex)=\hom_\Derived(\Xcomplex,\shift^n\Ycomplex)$. A \emph{triangle} in $\Derived$ will always mean a distinguished (exact) triangle. Given a map $f:X \to Y$ in $\Derived$ we denote by $\cone(f)$ any object which completes $X \xrightarrow{f} Y$ to a triangle.

All spectral sequences presented in this paper are of homological grading. This means that the differential on the $\Espec^r$ term is $d_r:\Espec^r_{p,q} \to \Espec^r_{p-r,q+r-1}$.

\subsection*{Acknowledgements}

I am indebted to J.P.C. Greenlees for many useful conversations and very helpful comments.

\section{Modules over a graded ring}
\label{sec: Modules over a graded ring}

The goals of this section are to recall the construction of the derived category of modules over a graded ring and to assure the reader that the possible sign conventions on grading shifts can be safely ignored. Here we also justify our use of the graded-commutative center. This section contains no new material and can be skipped without fear if one is familiar with graded modules over graded rings.

Let $R$ be an $\sphere$-algebra, then $\pi_* R$ is a graded ring and the functor $\pi_*$ sends $R$-modules to graded $\pi_* R$-modules. However, for an $R$-module $M$ the graded module $\pi_* \Sigma M$ is not simply a grading shift of $\pi_* M$. Indeed there is a sign twist involved in the action of $\pi_* R$ on $\pi_* \Sigma M$. Thus the grading shift functor we consider has a sign twist built into it.

\begin{definition}
\label{def: definition of the twisted grading shift functor}
Let $\ring$ be a graded ring and $\module$ a graded $\ring$-module. We define the \emph{twisted grading shift} of $\module$, denoted $\shift\module$, to be the graded $\ring$-module whose underlying graded abelian group is given by $(\shift\module)_n=\module_{n-1}$. The $\ring$-module structure on $\shift\module$ is given by
\[r \shift m = (-1)^{|r|}\shift (r m)\]
for $r\in \ring$ and $m \in \module$. We denote the standard grading shift functor by $\Sshift$.
\end{definition}

Note that if $\pi_* R$ is concentrated in even degrees, then there is no sign twist involved in the functor $\shift$ and it becomes the standard grading shift functor.

\begin{example}
If $\ring=\pi_* R$ for some $\sphere$-algebra $R$, then for any $R$-module $M$: $\pi_* \Sigma M = \shift \pi_* M$.
\end{example}

\begin{lemma}
There is a natural isomorphism $\Sshift \cong \shift$.
\end{lemma}
\begin{proof}
Assigning $Sm \mapsto (-1)^{|m|}\shift m$ for $m\in\module$ yields the desired isomorphism.
\end{proof}

It is the isomorphism $\Sshift \cong \shift$ which shows that our choice of using the twisted grading shift functor does not matter.

The category of $\ring$-modules has as objects graded $\ring$-modules and as morphisms $\ring$-homomorphisms of degree zero. For $\ring$-modules $\module$ and $\Nmodule$ we denote by $\hom_\ring(\module,\Nmodule)$ the abelian group of morphisms $\module \to \Nmodule$ in the category of $\ring$-modules. It is easy to see that this is an abelian category with enough projectives and injectives and one can construct its derived category $\Derived(\ring)$ in a standard fashion. For a more detailed discussion on the construction of $\Derived(\ring)$ see~\cite{YekutieliDualizing}.

\begin{definition}
\label{def: definition of graded Hom over a graded ring}
Let $\module$ and $\Nmodule$ be $\ring$-modules. Let $\Hom_\ring(\module,\Nmodule)$ be the graded abelian group
\[\Hom_\ring(\module,\Nmodule)\ing{n} = \hom_\ring(\shift^n \module,\Nmodule) = \hom_\ring(\module,\shift^{-n}\Nmodule)\]
The $\Hom_\ring(-,-)$ functor on $\ring$-modules defined above has a standard extension to $\ring$-chain complexes. As noted in Section~\ref{sec: Introduction}, we always assume to have taken a correct resolution before applying this functor.
\end{definition}

\begin{remark}
We remind the reader of our convention that $\hom_{\Derived(\ring)}^*(-,-)$ is in fact bigraded. Thus for $\ring$-chain complexes $\Xcomplex$ and $\Ycomplex$ we have
\[\hom_{\Derived(\ring)}^{p,q} (\Xcomplex,\Ycomplex) = \hom_{\Derived(\ring)}(\Xcomplex, \trans^p \shift^q \Ycomplex) \cong
(H^p\Hom_\ring(\Xcomplex,\Ycomplex))\ing{-q} \]
\end{remark}

\begin{remark}
Suppose $\Emodule$ is an injective $\pi_*R$-module, where $R$ is an $\sphere$-algebra. By Brown representability there exists an $R$-module $E$ and a natural isomorphism $\hom_{\Derived(R)}(M,E) \cong \hom_{\pi_* R}(\pi_* M, \Emodule)$ for every $R$-module $M$. The upshot of Definition~\ref{def: definition of graded Hom over a graded ring} above is that now the same natural isomorphism gives
\[ \pi_n\Hom_R(M,E) \cong \Hom_{\pi_*R}(\pi_* M,\Emodule)\ing{n} \]
\end{remark}

\begin{remark}
Suppose $x \in \ring$ satisfies $xy=(-1)^{|x||y|} yx$ for all $y\in \ring$. Then it is a simple exercise to see that multiplication by $x$ yields an $\ring$-module map $\shift^{|x|} \module \to \module$ for an $\ring$-module $\module$. This justifies calling such an element ``central''.
\end{remark}

\section{Cellular approximation and nullification}
\label{sec: Cellular approximation and nullification}

The key ingredient of this paper is the notion of cellular approximation. It is therefore worthwhile to set aside a section on its definition and properties. We shall follow the definitions set by Dwyer and Greenlees in~\cite{DwyerGreenlees}, given here in the setting of a general triangulated category. This allows for a unified treatment of both situations we have in mind, namely the derived category of an $\sphere$-algebra and the derived category of a graded ring.

Recall that a full subcategory $\cc$ of a triangulated category $\Derived$ is called a \emph{localizing} subcategory if $\cc$ is closed under completion of triangles and arbitrary coproducts. In the case of the derived category of a graded ring we also demand that a localizing subcategory be closed under the twisted grading shift functor $\shift$ . The localizing subcategory generated by a class of objects $\Aclass \subset \Derived$ is the minimal localizing subcategory of $\Derived$ among all localizing subcategories which contain $\Aclass$. We denote this localizing subcategory by $\loc{\Aclass}$.

\begin{definition}
\label{def: Cellularization}
Let $\Derived$ be a triangulated category and let $\Aclass$ be a class of objects in $\Derived$. A morphism $f:X\to Y$ in $\Derived$ is called an \emph{$\Aclass$-equivalence} if for every $A\in\Aclass$ the morphism $\hom_\Derived^*(A,f)$ is an isomorphism. An object $X\in \Derived$ is \emph{$\Aclass$-cellular} if $X\in\loc{\Aclass}$. We say $C$ is an \emph{$\Aclass$-cellular approximation of $X$}, if $C$ is $\Aclass$-cellular and there is an $\Aclass$-equivalence $\eta:C \to X$. An $\Aclass$-cellular approximation of $X$ shall be denoted by $\cell_\Aclass X$. If $\Derived=\Derived(R)$ for some $\sphere$-algebra or graded ring $R$ then we denote an $\Aclass$-cellular approximation of $X$ by $\cell_\Aclass^R X$ whenever we need to emphasize the category we work over.
\end{definition}

Clearly, $\Aclass$-cellular approximation depends only on the localizing category generated by $\Aclass$, i.e. if $\loc{\Aclass}=\loc{\Bclass}$ then $\cell_\Aclass X \cong \cell_\Bclass X$ whenever such cellular approximation exists. A similar remark applies for $\Aclass$-equivalences. This implies that the $\Aclass$-equivalence $\cell_{\Aclass} X \to X$ is an $\loc{\Aclass}$-colocalization of $X$. To make the terminology a bit more convenient we make the following definition.

\begin{definition}
Let $\Aclass$ be a class of objects in $\Derived$. We say that a map $\eta:C \to X$ is an \emph{$\Aclass$-colocalization of $X$} if $\eta$ is an $\loc{\Aclass}$-colocalization of $X$. Thus $C$ is $\cell_\Aclass X$ and $\eta$ is an $\Aclass$-equivalence.
\end{definition}

The counterpart of cellular approximation is \emph{nullification}, described below.
\begin{definition}
\label{def: Nullification}
An object $X\in \Derived$ is called \emph{$\Aclass$-null} if $\hom_\Derived^*(A,X)=0$ for every $A\in\Aclass$. An object $N$ is an \emph{$\Aclass$-nullification of $X$} if there is a morphism $\nu:X \to N$  which is initial among maps in $\Derived$ from $X$ to $\Aclass$-null objects. This is equivalent to saying that $N$ is $\Aclass$-null and every cone of $\nu$ is $\Aclass$-cellular. The morphism $\nu$ shall be called an \emph{$\Aclass$-nullification map} of $X$. We use the notation $\Null_\Aclass X$ for an $\Aclass$-nullification of $X$ in $\Derived$.
\end{definition}

\begin{remark}
It is worthwhile mentioning the accepted notation in algebra, which is in fact more convenient in this case. The class of $\Aclass$-null objects is commonly denoted $\Aclass^\perp$. Assuming that $\Aclass$-cellular approximation exists for every object in $\Derived$, it is then well known that the $\Aclass$-cellular objects are the objects $C$ satisfying $\hom_\Derived^*(C,N)=0$ for every $N\in \Aclass^\perp$; in this case the class of $\Aclass$-cellular objects can be denoted by $^\perp(\Aclass^\perp)$.
\end{remark}

It is easy to see that for any $X$ there is a distinguished triangle:
\[ \cell_\Aclass X \xrightarrow{\eta} X \xrightarrow{\nu} \Null_\Aclass X\]
where $\eta$ is an $\Aclass$-equivalence map of $X$ and $\nu$ is an $\Aclass$-nullification map of $X$. Hence the existence of $\Aclass$-cellular approximation implies the existence of $\Aclass$-nullification and vice versa.

In general, $\Aclass$-cellular approximations of objects need not exist. However, when $\loc{\Aclass}$ is generated by a set of objects and $\Derived$ is the derived category of an $\sphere$-algebra (or a graded ring) then $\Aclass$-cellular approximations exist for every object in $\Derived$. This follows from~\cite{Hirschhorn} (see also~\cite{FarjounBook} for cellular approximation in a topological setting). Note that if $\cell_\Aclass X$ can be formed for every $X$, then an object $C$ is $\Aclass$-cellular if and only if $\hom_\Derived^*(C,N)=0$ for every $\Aclass$-null object $N$ (see for example \cite{DwyerGreenlees}).

The following proposition records for future reference well known equivalent definitions of $\Aclass$-cellular approximation. Its proof is trivial and therefore omitted.
\begin{proposition}
\label{pro: Equivalent conditions for Cell and Null}
For a morphism $\eta:C \to X$ in $\Derived$ the following are equivalent:
\begin{enumerate}
  \item $\eta$ is an $\Aclass$-equivalence and $C$ is $\Aclass$-cellular,
  \item the morphism $X \to \cone(\eta)$ is an $\Aclass$-nullification map,
  \item $C$ is $\Aclass$-cellular and $\cone(\eta)$ is $\Aclass$-null.
\end{enumerate}
\end{proposition}

We also record a well known principle of recognizing $\Aclass$-cellular objects. We give this principle in both in the setting of a derived category of an $\sphere$-algebra (Lemma~\ref{lem: Bounded above module spectrum is cellular pi_0-cellular}) and in that of a graded ring (Lemma~\ref{lem: Bounded above chain comples is cellular}). The proof of this principle, applicable for both cases, can be found for example in~\cite{ShamirEMSS}.

Recall that if $R$ is a connective $\sphere$-algebra then one can form an Eilenberg-Mac Lane spectrum $H\pi_0(R)$ and a map of $\sphere$-algebras $R \to H\pi_0(R)$ realizing the isomorphism on $\pi_0$ of both $\sphere$-algebras (see~\cite[IV.3]{EKMM}).
\begin{lemma}
\label{lem: Bounded above module spectrum is cellular pi_0-cellular}
Suppose $R$ is a connective $\sphere$-algebra. Let $M$ be a bounded-above $R$-module, then $M$ is $\{H\pi_nM\}_{n\in\Int}$-cellular.
\end{lemma}

\begin{lemma}
\label{lem: Bounded above chain comples is cellular}
Suppose $\ring$ is a graded ring and $\Aclass$ is a class of $\ring$-modules. If $C$ is an $\ring$-complex such that
\begin{enumerate}
\item $C$ is bounded above and
\item $H_n C$ is an $\Aclass$-cellular $\ring$-module for all $n$.
\end{enumerate}
Then $C$ is an $\Aclass$-cellular complex.
\end{lemma}

\section{Colocalization with respect to a hereditary torsion theory}
\label{sec: Colocalization with respect to a hereditary torsion theory}

In order to construct the spectral-sequence we require several of the results of~\cite{ShamirTorsion}, which deal with colocalization with respect to a hereditary torsion theory. Note that in~\cite{ShamirTorsion} the setting considered was that of modules over a ring, while the setting needed here is that of graded modules over a graded ring. Nevertheless, the results (and proofs) of~\cite{ShamirTorsion} apply also to the graded case. Throughout this section $\ring$ denotes a graded ring.

We begin by recalling the definition and some properties of hereditary torsion theories. These do not change when passing to the setting of graded modules over a graded ring. A comprehensive treatment of this subject can be found, for example, in~\cite{Stenstrom}.
\begin{definition}
\label{def: Hereditary torsion theory}
A \emph{hereditary torsion class} $\tc$ is a class of $\ring$-modules that is closed under submodules, quotient modules, coproducts and extensions. Since we are working in a graded setting, we shall also require a hereditary torsion class to be closed under the twisted grading shift functor $\shift$. Closure under extensions means that if $0 \to \module_1 \to \module_2 \to \module_3 \to 0$ is a short exact sequence with $\module_1$ and $\module_3$ in $\tc$, then so is $\module_2$. The modules in $\tc$ will be called \emph{$\tc$-torsion modules} (or just torsion modules when the torsion theory is clear from the context). The class of \emph{torsion-free} modules $\ff$ is the class of all modules $\module$ satisfying $\ext_\ring^0(\Cmodule,\module)=0$ for every $\Cmodule \in \tc$. The pair $(\tc,\ff)$ is referred to as a \emph{hereditary torsion theory}.

To every hereditary torsion theory $(\tc,\ff)$ there is an \emph{associated radical} $t$. Given an $\ring$-module $\module$, the module $t(\module)$ is the largest torsion submodule of $\module$. Because $\tc$ is hereditary, the module $\module/t(\module)$ is a torsion-free module.

An $\ring$-module $\Emodule$ is called a \emph{cogenerator} for a hereditary torsion class $\tc$ if $\tc$ is the class of modules $\module$ such that $\Hom_\ring(\module,\Emodule)=0$. As in the non-graded case, every hereditary torsion class $\tc$ has an injective cogenerator, see Lemma~\ref{lem: injective cogenerator for a hereditary torsion theory} below.
\end{definition}

\begin{lemma}
\label{lem: injective cogenerator for a hereditary torsion theory}
Every hereditary torsion class has an injective cogenerator.
\end{lemma}
\begin{proof}
The proof is almost the same as the one given in \cite[VI.3.7]{Stenstrom} for the non-graded case. A minor change in the definition of the injective cogenerator is needed because of the requirement that $\tc$ be closed under twisted grading shift.

Let $(\tc,\ff)$ be a hereditary torsion theory. For an injective cogenerator of $\tc$ one can take the injective module $\Emodule=\prod \shift^n\Emodule(\ring/\aideal)$, where $n\in \Int$, $\aideal$ goes over all left ideals of $\ring$ such that $\ring/\aideal$ is torsion-free and $\Emodule(\ring/\aideal)$ is the injective hull of $\ring/\aideal$. The rest of the proof, showing that $\Emodule$ is an injective cogenerator for $\tc$, is the same as in~\cite{Stenstrom}.
\end{proof}

The main example we have in mind is the following.
\begin{example}
\label{exa: Hereditary torsion generated by Ring/Ideal}
Suppose $\ideal \subset \ring$ is a two-sided ideal of $\ring$ that is finitely generated as a left $\ring$-module. Recall an $\ring$-module $\module$ is an \emph{$\ideal$-power torsion} module if for every $m\in \module$ there exists some $n$ such that $\ideal^n m=0$. It is not difficult to show that the class $\tc$ of $\ideal$-power torsion modules is a hereditary torsion class, see~\cite[VI.6.10]{Stenstrom}. The radical of $\tc$ is given by $t(\module)=\{m\in \module\,|\,\ideal^n m=0 \text{ for some } n\}$ and the quotient $\module/t(\module)$ has no $\ideal$-power torsion elements.
\end{example}

We record the following results from~\cite{ShamirTorsion}, these pass to the graded case with no change.
\begin{proposition}[{\cite[Lemma 2.8]{ShamirTorsion}}]
\label{pro: Hereditary torsion class generated by cyclics}
Let $\tc$ be a hereditary torsion class and denote by $\Cyc_\tc$ the set of all cyclic $\tc$-torsion modules. Then $\loc{\tc}=\loc{\Cyc_\tc}$ and therefore $\tc$-colocalization and $\tc$-nullification exist for every $\ring$-complex.
\end{proposition}

\begin{proposition}[{\cite[Corollary 2.11]{ShamirTorsion}}]
\label{pro: Bounded above is Cellular}
Let $\tc$ be a hereditary torsion class, then a bounded-above $\ring$-complex $\Xcomplex$ is $\tc$-cellular if and only if $H_n(\Xcomplex)$ is $\tc$-torsion for all $n$.
\end{proposition}

We shall also require the following result.
\begin{lemma}
\label{lem: Cellularization at R/I is I-power torsion when}
Let $\ring \to \Sring$ be a surjection of graded rings and let $\ideal$ be the kernel of this map. Suppose that the class of $\ideal$-power torsion modules, denoted $\tc$, is a hereditary torsion class. Then $\loc{\tc}=\loc{\Sring}$ and therefore $\cell_\tc \simeq \cell_\Sring$.

\end{lemma}
\begin{proof}
Since $\Sring \in \loc{\tc}$, we need only show that $\loc{\tc} \subset \loc{\Sring}$. As noted in Proposition~\ref{pro: Hereditary torsion class generated by cyclics}, $\loc{\tc}$ is generated by the cyclic $\tc$-torsion modules.  Any cyclic $\tc$-torsion module is a cyclic $\ring/\ideal^n$-module for some $n$. Hence it is enough to show that every $\ring/\ideal^n$-module is $\Sring$-cellular. It is easy to see that every $\ring/\ideal^n$-module is $\ring/\ideal^n$-cellular and therefore it suffices to show that $\ring/\ideal^n$ is $\Sring$-cellular for every $n$. There are short exact sequences
\[0 \to \ideal^{n-1}/\ideal^{n} \to \ring/\ideal^n \to \ring/\ideal^{n-1}\to 0.\]
Noting that $\ideal^{n-1}/\ideal^{n}$ is an $\Sring$-module and hence $\Sring$-cellular, and that $\loc{\Sring}$ is closed under triangles, it is easy to see that an inductive argument completes the proof.
\end{proof}

\section{Explicit construction of $\tc$-nullification}
\label{sec: Explicit construction of tc-nullification}

Let $\ring$ be a graded ring and let $(\tc,\ff)$ be a hereditary torsion theory on $\ring$-modules with an associated radical $t$. In~\cite{ShamirTorsion} there is described a construction for $\tc$-nullification using torsion-free injective modules (this construction is a generalization of a construction of Benson from~\cite{Benson}). The construction from~\cite{ShamirTorsion} is recorded below, with a minor modification.

\begin{construction}
\label{con: Nullification Construction}
Given an $\ring$-module $\module$, we construct a cochain complex \, $\nn=\nn^0 \xrightarrow{d^0} \nn^1 \xrightarrow{d^1} \nn^2 \xrightarrow{d^2} \cdots$\, inductively. Let $\module^0=\module$, let $\Gmodule^0 = \module^0/t(\module^0)$ and let $\nn^0$ be the injective hull of $\Gmodule^0$. Proceed by induction, thus
\begin{align*}
\module^{n+1} &= \nn^n/\Gmodule^n,\\
\Gmodule^{n+1} &= \module^{n+1}/t(\module^{n+1}) \text{ and }\\
\nn^{n+1} & \text{ is the injective hull of } \Gmodule^{n+1}.
\end{align*}
The differential $d^n:\nn^n \to \nn^{n+1}$ is the composition of the obvious morphisms $\nn^n \to \module^{n+1} \to \Gmodule^{n+1} \to \nn^{n+1}$. Note that the obvious morphism $\module \to \nn^0$ induces a map $\nu:\module \to \nn$ of $\ring$-complexes.
\end{construction}

\begin{lemma}[{\cite[Lemma 3.2]{ShamirTorsion}}]
\label{lem: Construction of tc-nullification}
The complex $\nn$ constructed in~\ref{con: Nullification Construction} is a $\tc$-nullification of $\module$ and the map $\nu: \module \to \nn$ is a $\tc$-nullification map of $\module$.
\end{lemma}
\begin{proof}
Every $\nn^n$ is the injective hull of a torsion-free module and therefore $\nn^n$ itself is torsion-free. Hence, $\nn$ is a $\tc$-null complex. Now consider the cochain complex $\Ccomplex$, given by:
\[ \module \to \nn^0 \to \nn^1 \to \nn^2 \to \cdots \]
It is easy to see that $H^n(\Ccomplex)=t(\module^n)$ for all $n\geq 0$. This implies that $\Ccomplex$ is $\tc$-cellular (Lemma~\ref{lem: Bounded above chain comples is cellular}).

There is a triangle $\Ccomplex \to \module \xrightarrow{\nu} \nn$. Since $\Ccomplex$ is $\tc$-cellular and $\nn$ is $\tc$-null then, by Proposition~\ref{pro: Equivalent conditions for Cell and Null}, $\Ccomplex$ is a $\tc$-cellular approximation of $\module$ and $\nn$ is a $\tc$-nullification of $\module$.
\end{proof}

One can use Lemma~\ref{lem: Construction of tc-nullification} alone to construct the desired spectral sequence. However, the following trick from~\cite{ShamirTorsion} will make the construction of the spectral sequence immediate.
\begin{theorem}[{\cite[Theorem 2.2]{ShamirTorsion}}]
\label{the: Nullification as double endomorphism}
Let $\tc$ be a hereditary torsion class. For every $\ring$-module $\module$ there exists a torsion-free injective $\ring$-module $\ee$ such that the natural map of complexes
\[\module \to \Hom_{\End_\ring(\ee)}(\Hom_\ring(\module,\ee),\ee)\]
is a $\tc$-nullification map of $\module$.
\end{theorem}

It will sometimes be necessary need to specify the injective module used in Theorem~\ref{the: Nullification as double endomorphism}. For that purpose we record one more result from~\cite{ShamirTorsion}.
\begin{proposition}[{\cite[Proposition 3.5]{ShamirTorsion}}]
\label{pro: resolution by fg projectives for nullification}
Let $\module$ be an $\ring$-module and let $\ee$ be an injective cogenerator of the hereditary torsion theory $\tc$. Denote by $\Qring$ the graded ring $\End_\ring(\ee)$. If the $\Qring$-module $\Hom_\ring(\module,\ee)$ has a resolution composed of finitely generated projective modules in each degree, then $\Hom_\Qring(\Hom_\ring(\module,\ee),\ee)$ is the $\tc$-nullification of $\module$.
\end{proposition}

\section{Construction of the spectral sequences}
\label{sec: Construction of the spectral sequences}

We will, in fact, construct two spectral sequences. First we will construct a nullification spectral sequence and from it deduce a colocalization spectral sequence. But before constructing the spectral sequences we must take a short detour and review a certain adjunction.

Let $R$ be an $\sphere$-algebra and $E$ be an $R$-module. Denote by $\Endring$ the $\sphere$-algebra $\End_R(E)$. The module $E$ is in fact a left $R \otimes_{\sphere} \Endring$-module. There are two functors:
\[ \Hom_R(-,E):\Derived(R)^\mathrm{op} \to \Derived(\Endring) \quad \text{ and } \quad \Hom_\Endring(-,E):\Derived(\Endring)^\mathrm{op} \to \Derived(R).\]
That $\Hom_R(M,E)$ is an $\Endring$-module comes from the extra left $\Endring$-action on $E$ (this is implicit in \cite[III.6.5]{EKMM}). The exact same argument applies for the functor $\Hom_\Endring(-,E)$.

\begin{definition}
For an $R$-module $M$, the \emph{$E$-dual} of $M$ is the $\Endring$-module $\Hom_R(M,E)$. The \emph{$E$-double dual} of $M$ is the $R$-module
\[\Hom_\Endring(\Hom_R(M,E),E)\]
\end{definition}

\begin{lemma}
There is a natural transformation $1 \to \Hom_\Endring(\Hom_R(-,E),E)$.
\end{lemma}
\begin{proof}
Let $U$ be an $\Endring$-module, then:
\begin{align*}
\hom_R(M,\Hom_\Endring(U,E))&\cong \hom_{R\otimes_\sphere \Endring}(M\otimes_\sphere U,E) \\
&\cong \hom_{\Endring\otimes_\sphere R}(U\otimes_\sphere M,E)\\
&\cong \hom_\Endring(U,\Hom_R(M,E)).
\end{align*}
The first and last isomorphisms are the ones described in \cite[III.6.5(i)]{EKMM}. This leads to an adjunction:
\[ \xymatrixcompile{ {\Derived(R)} \ar@/_/[r]_F
            & {\Derived(\Endring)^{\text{op}}} \ar@/_/[l]_G }, \]
where $F$ is $\Hom_R(-,E)$ and $G$ is $\Hom_\Endring(-,E)$. The natural transformation mentioned above is simply the unit of this adjunction.
\end{proof}

Note that the image of the functor $\Hom_\Endring(-,E):\Derived(\Endring) \to \Derived(R)$ is contained in the colocalizing subcategory generated by $E$. The \emph{colocalizing subcategory} generated by $E$ is the smallest triangulated subcategory closed under retracts, isomorphisms, completion of triangles and products. Also note that if $E$ is $A$-null for some $R$-module $A$, then every object in the colocalizing subcategory generated by $E$ is $A$-null.

We are now ready to construct the nullification spectral sequence. Given a hereditary torsion class $\tc$ on $\pi_* R$-modules recall that an $R$-module $M$ is $\tc$-cellular if $\pi_*M \in \tc$ and that $M$ is $\tc$-null if $\Hom_R(C,M)\simeq 0$ for every $\tc$-cellular $R$-module $C$.

\begin{lemma}
\label{lem: Spectral sequence for nullification}
Let $R$ be an $\sphere$-algebra and let $\tc$ be a hereditary torsion class on $\pi_* R$-modules. For every $R$-module $M$ there exists a map $\nu:M \to N$ such that $N$ is a $\tc$-null $R$-module and there is a spectral sequence of $\pi_*R$-modules:
\[ E^2_{p,q} = H_{p,q}(\Null_\tc^{\pi_*R}\pi_*M) \ \Rightarrow \ \pi_{p+q}(N)\]
This spectral sequence has conditional convergence.
\end{lemma}
\begin{proof}
Let $\nn$ be the complex described in the Nullification Construction~\ref{con: Nullification Construction} for the $\pi_*R$-module $\pi_*M$. Let $\ee$ be an appropriate injective $\pi_*R$-module as in Theorem~\ref{the: Nullification as double endomorphism}. To be specific, we choose $\ee$ to be the product $\prod_i \nn^i$. It follows from~\cite{ShamirTorsion} that $\ee$ is an appropriate injective, i.e.
\[\Null_\tc^{\pi_*R}\pi_*M \simeq \Hom_{\End_{\pi_*R}(\ee)}(\Hom_{\pi_*R}(\pi_*M,\ee),\ee)\]

Brown representability implies that there exists an $R$-module $E$ such that there is a natural isomorphism $\pi_*\Hom_R(X,E) \cong \Hom_{\pi_*R} (\pi_*X,\ee)$ for every $R$-module $X$. Let $N$ be the $E$-double dual of $M$, i.e.
\[ N=\Hom_{\End_R(E)}(\Hom_R(M,E),E).\]
The $R$-module $E$ is certainly $\tc$-null and since $N$ is in the colocalizing subcategory generated by $E$, $N$ is also be $\tc$-null.

In~\cite[IV.4]{EKMM}, a spectral sequence for calculation of $\ext_{\End_R(E)}(\Hom_R(M,E),E)$ is presented. This spectral sequence has the form:
\[ E_2^{p,q}=\ext_{\pi^*\End_R(E)}^{p,q}(\pi^*\Hom_R(M,E),\pi^*E) \ \Rightarrow \ \pi^{p+q}N\]
The isomorphisms:
\[\End_R(E)_* \cong \End_\ring(\ee,\ee), \ \Hom_R(M,E)_*\cong \Hom_\ring(M_*,\ee)\text{ and } E_*\cong \ee\]
complete the proof.
\end{proof}

\begin{remark}
\label{rem: Generality of the spectral sequence}
The use of the spectral sequence from~\cite[IV.4]{EKMM} is just a convenient short cut. We could just as well have lifted the complex of torsion-free injectives constructed in Nullification Construction~\ref{con: Nullification Construction} to a filtration of $M$. This process would result in the same spectral sequence as that of~\cite[IV.4]{EKMM}. It follows that Lemma~\ref{lem: Spectral sequence for nullification} above holds with $\Derived(R)$ replaced by any triangulated category with a compact generator, as do other results in this paper.
\end{remark}

We shall see that any reasonably constructed nullification spectral sequence would induce a colocalization spectral sequence. But we must qualify the term ``reasonable'' first.
\begin{definition}
\label{def: Reasonable nullification spectral sequence}
Let $M$ be an $R$-module and let $\tc$ be a hereditary torsion class on $\pi_* R$-modules. A \emph{reasonable} $\tc$-nullification spectral sequence for $M$ consists of the following data:
\begin{enumerate}
\item a $\tc$-null module $E^0$ and a map $\nu:M \to E^0$, the module $E^0$ will be called the \emph{nullification candidate}.
\item Distinguished triangles $\Sigma^{-p} N^p \xleftarrow{k^p} E^p \xleftarrow{i^p} E^{p+1} \xleftarrow{j^{p+1}} \Sigma^{-p-1} N^p$.
\end{enumerate}
The $\Espec^1$-term of the resulting spectral sequence induced by the tower $E^0 \leftarrow E^1 \leftarrow\cdots$ is the cochain complex of $\pi_*R$-modules: $\pi_*N^0 \xrightarrow{k^1j^0} \pi_*N^1 \xrightarrow{k^2j^1} \pi_*N^2 \to \cdots$. We require that the map of $\pi_*R$-complexes $\pi_*M \to \Espec^1$ (induced by $M \xrightarrow{\nu} E^0 \xrightarrow{k^0} N^0$) be a $\tc$-nullification of $\pi_*M$. This in particular implies that the $\Espec^2_{p,q}$-term is $H_{p,q} (\Null_{\tc}^{\pi_*R} \pi_*M)$.
\end{definition}

\begin{lemma}
\label{lem: Colocalization spectral sequence induced by nullification}
Let $R$, $\tc$ and $M$ be as in Lemma~\ref{lem: Spectral sequence for nullification} above. Given a reasonable $\tc$-nullfication spectral sequence for $M$ there exists an induced $\tc$-colocalization spectral sequence for $M$. If the nullfication spectral sequence conditionally converges to the homotopy groups of the nullification candidate $E^0$ then the colocalization spectral sequence conditionally converges to the homotopy groups of a module $C^0$ and there is a distinguished triangle $C^0 \to M \xrightarrow{\nu} E^0$.
\end{lemma}
\begin{proof}
Let $C^0$ be the homotopy fiber of $M \to E^0$. For $p>0$ set $C^p=\Sigma^{-1}E^{p-1}$. Thus there are triangles
\[ M \xleftarrow{k^{-1}} C^0 \xleftarrow{i^{-1}} C^1 \xleftarrow{j^0} \Sigma^{-1} M \quad \text{and } \]
\[\Sigma^{-p-1}N^{p-1} \leftarrow C^p \leftarrow C^{p+1} \leftarrow \Sigma^{-p-2} N^{p-1} \quad \text{ for }p>0. \]

Define $\Dspec^1_{p,q}=\pi_{p+q}(C^{-p})$, $\Espec^1_{0,q}=\pi_{q}(M)$ and $\Espec^1_{p,q}=\pi_{q}(E^{-p-1})$ for $p<0$. The distinguished triangles of Definition~\ref{def: Reasonable nullification spectral sequence} yield the following maps:
\begin{align*}
i &: \Dspec^1_{p,q} \to \Dspec^1_{p+1,q-1} \\
j &: \Espec^1_{p,q} \to \Dspec^1_{p-1,q}\\
k &: \Dspec^1_{p,q} \to \Espec^1_{p,q}
\end{align*}
which make $\Dspec$ and $\Espec$ into an exact couple. It is a simple matter to see that $\Espec^1_{-p,*}$ with the map $\delta=k\circ j$ is the cochain-complex:
\[ \pi_*M \to \pi_*N^0 \to \pi_*N^1 \to \cdots \to \pi_*N^{p-1} \xrightarrow{\delta} \pi_*N^p \to \cdots\]
This cochain complex is clearly $\cell_\tc^{\pi_*R} \pi_*M$, thus yielding the desired $\Espec^2$-page.

To show conditional convergence we consider the following triangle of towers:
\[ \xymatrixcompile{
{C^0} \ar[d] & \ar[l] {C^1} \ar[d] & \ar[l] {C^2} \ar[d] & \ar[l] \cdots \\
{M}   \ar[d] & \ar[l]  {0}  \ar[d] & \ar[l]  {0}  \ar[d] & \ar[l] \cdots \\
{E^0}        & \ar[l]^{=} {E^0}    & \ar[l] {E^1}        & \ar[l] \cdots } \]
The homotopy limit of the bottom tower is equivalent to the homotopy limit of the tower:
\[ E^0 \leftarrow E^1 \leftarrow E^2 \leftarrow \cdots \]
which is zero since we assume conditional convergence of the nullification spectral sequence. The homotopy limit of the middle tower is clearly zero, hence the homotopy limit of the top tower is also zero. As in~\cite[IV.5]{EKMM}, this implies that:
\[ \mathrm{lim}(\Dspec^1_{0,*} \leftarrow \Dspec^1_{-1,*} \leftarrow \cdots ) =0\]
and
\[ \mathrm{lim}^1(\Dspec^1_{0,*} \leftarrow \Dspec^1_{-1,*} \leftarrow \cdots ) =0,\]
resulting in conditional convergence as defined in the work of Boardman~\cite{BoardmanCCSpecSeq}.
\end{proof}

\begin{corollary}
\label{cor: The colocalization spectral sequence}
Let $R$, $\tc$, $M$ and $N$ be as in Lemma~\ref{lem: Spectral sequence for nullification} above and let $C$ be the homotopy fiber of the map $\nu:M \to N$. Then the spectral sequence constructed in Lemma~\ref{lem: Spectral sequence for nullification} is a reasonable nullification spectral sequence and hence induces a colocalization spectral sequence
\[ \Espec^2_{p,q} = H_{p,q}(\cell_\tc^{\pi_*R}\pi_*M) \ \Rightarrow \ \pi_{p+q}(C)\]
which has conditional convergence.
\end{corollary}
\begin{proof}
As before, let $\nn$ is the complex described in the Nullification Construction~\ref{con: Nullification Construction} for the $\pi_*R$-module $\pi_*M$. Let $N^p$ be the lift of $\nn^p$ given by brown representability and let $E$ be the product $\prod_i N^p$.

Following the details of the construction of the $\ext$-spectral sequence in~\cite[IV.5]{EKMM} we see there are triangles:
\[ \Sigma^{-p} N^p \xleftarrow{k^p} E^p \xleftarrow{i^p} E^{p+1} \xleftarrow{j^{p+1}} \Sigma^{-p-1} N^p\]
for $p\geq0$. These fiber sequences satisfy the following properties:
\begin{enumerate}
\item $\pi_*(N^p)\cong \nn^p$.
\item $E^0 \simeq \Hom_{\Endring}(\Hom_R(M,E),E)$ - this is the module $N$ from Lemma~\ref{lem: Spectral sequence for nullification} above.
\item The composition $k^{p+1} \circ j^p$ realizes the differential $\delta^{p+1}$.
\item The homotopy limit of the tower $E^0 \leftarrow E^1 \leftarrow \cdots$ is zero (\cite[IV.5]{EKMM}).
\item The composition $M \to E^0 \to N^0$ realizes the map $\pi_*M \to \nn^0$ in the Nullification Construction~\ref{con: Nullification Construction}.
\end{enumerate}
The only property needing explanation here is the last one. In choosing the specific module $E$ above we have ensured that $N^0$ is equivalent to its double dual, simply because $N^0$ is a retract of $E$. Now, let $f:M \to N^0$ be such that $\pi_*f:\pi_*M \to \nn^0$ is the morphism described in the Nullification Construction~\ref{con: Nullification Construction}. Then applying the $E$-double dual functor to $f$ yields a commutative diagram:
\[ \xymatrixcompile{ {M} \ar[r]^f \ar[d] & {N^0} \ar[d]^\simeq \\ {N} \ar[r] & {\Hom_{\Endring}(\Hom_R(N^0,E),E)} }\]
which proves the last property. Lemma~\ref{lem: Colocalization spectral sequence induced by nullification} completes the proof.
\end{proof}

\section{Proofs of Theorems \ref{thm: First Theorem}, \ref{thm: Second Theorem} and \ref{thm: Third Theorem}}
\label{sec: Proofs}

\begin{proof}[Proof of Theorem~\ref{thm: First Theorem}]
To obtain a spectral sequence $\Espec^r_{p,q}$ with the correct $\Espec^2$-page we need only invoke Corollary~\ref{cor: The colocalization spectral sequence}. Hence there is a triangle $C \to M \to N$, with $N$ being $\tc$-null and the spectral sequence of Corollary~\ref{cor: The colocalization spectral sequence} converges conditionally to $\pi_*C$.

Now suppose the $\pi_*R$-complex $\cell_\tc^{\pi_*R}(\pi_*M)$ is bounded below. This immediately implies that the spectral sequence has strong convergence \cite[Theorem 7.1]{BoardmanCCSpecSeq}. It remains to show that $C$ is $\cell_\tc M$. By Proposition~\ref{pro: Equivalent conditions for Cell and Null}, it is enough to show that $C$ is $\tc$-cellular. Equivalently, we must show that $\pi_*C$ is $\tc$-cellular.

By Proposition~\ref{pro: Bounded above is Cellular}, for every $p$ the $\pi_*R$-module $\Espec^2_{p,*}$ is $\tc$-torsion. Hence $E^r_{p,*}$ is $\tc$-torsion for every $p$ and $r$, because $\tc$ is closed under kernels and cokernels. Since $\cell_\tc^{\pi_*R} (\pi_*M)$ is bounded below, the spectral sequence collapses for some $r$. We see that $\pi_*C$ has a finite filtration whose successive quotients are $\tc$-torsion. Because $\tc$ is also closed under extensions, $\pi_*C$ is $\tc$-torsion.
\end{proof}

\begin{proof}[Proof of Theorem~\ref{thm: Second Theorem}]
It is enough to show that the condition of Theorem~\ref{thm: First Theorem} is satisfied. Namely that the $\pi_*R$-complex $\cell_\tc^{\pi_*R} (\pi_*M)$ is bounded-below. If $\pi_*R$ is graded-commutative then this is a result of Dwyer and Greenlees~\cite[Proposition 6.10]{DwyerGreenlees}. If $\pi_*R$ is left Noetherian this follows from~\cite{ShamirGradedColocalization}. 
\end{proof}

\begin{proof}[Proof of Theorem~\ref{thm: Third Theorem}]
It should be pointed out that, since $\pi_*R$ is left Noetherian, $\ideal$ is finitely generated as a left $\pi_*R$-module. This implies that the class of $\ideal$-power torsion modules $\tc$ is a hereditary torsion class, see~\cite[VI.6.10]{Stenstrom}. It also implies that $\cell_\tc^{\pi_*R}(\pi_*M) \simeq \cell_k^{\pi_*R}(\pi_*M)$, by Lemma~\ref{lem: Cellularization at R/I is I-power torsion when}.

As in the proof of Theorem~\ref{thm: First Theorem}, we obtain the desired spectral sequence $\Espec^r_{p,q}$ by invoking Corollary~\ref{cor: The colocalization spectral sequence}. Thus there is a triangle $C \to M \to N$, with $N$ being $\tc$-null and the colocalization spectral sequence of  Corollary~\ref{cor: The colocalization spectral sequence} converges conditionally to $\pi_*C$. By the results of~\cite{ShamirGradedColocalization}
there exists $q_0$ such that $\Espec^2_{p,q}=0$ for all $q>q_0$. Strong convergence now follows from Boardman \cite[Theorem 7.1]{BoardmanCCSpecSeq}.

Our next task is to show that $C$ is $\cell_\tc^R M$. As in the proof of Theorem~\ref{thm: First Theorem}, the $\pi_*R$-modules $E^r_{p,*}$ are $\tc$-torsion for every $r$. Note that $E^r_{p,q}$ is a $\pi_*R$-module that is a sub-quotient of $E^r_{p,*}$. For example $E^r_{p,q_0}$ is a submodule of $E^r_{p,*}$ and $E^r_{p,q_0-1}$ is a submodule of the quotient $E^r_{p,*}/E^r_{p,q_0}$. Hence $E^r_{p,q}$ is $\tc$-torsion for all $r$,$p$ and $q$.

It is clear that for every $n$ there exists $r(n)$ such that $E^{r(n)}_{p,q}=E^\infty_{p,q}$ whenever $p+q=n$. Moreover, there are only finitely many pairs $(p,q)$ such that $p+q=n$ and $E^{r(n)}_{p,q}\neq 0$. Hence, $\pi_n C$ has a finite filtration whose successive quotients are $\tc$-torsion. We conclude that $\pi_n C$ is $\tc$-torsion for all $n$.

By Lemma~\ref{lem: Cellularization at R/I is I-power torsion when}, $\pi_n C$ is $k$-cellular as a $\pi_* R$-module. We claim this implies $\pi_n C$ is $k$-cellular also as a $\pi_0 R$-module. Consider the map of graded algebras $f:\pi_0 R\to \pi_*R$. We know that $\pi_n(C)$ is $k$-cellular over $\pi_*R$, this can be pulled back along $f$ to show that $\pi_n(C)$ is $k$-cellular over $\pi_0R$.

We see that the Eilenberg-Mac Lane spectrum $H(\pi_n C)$ is an $Hk$-cellular $H\pi_0 R$-module. This can be pulled back along the map $R \to H\pi_0 R$ to show that $H(\pi_n C)$ is $Hk$-cellular as an $R$-module. Hence, by Lemma~\ref{lem: Bounded above module spectrum is cellular pi_0-cellular}, $C$ is $Hk$-cellular. Since $Hk$ is $\tc$-cellular then $N$ is $Hk$-null and $C$ is $\tc$-cellular. Using Proposition~\ref{pro: Equivalent conditions for Cell and Null} we conclude that $C$ is both $\cell_{Hk}^R M$ and $\cell_\tc^R M$.
\end{proof}

\section{Universality of the Nullification Spectral Sequence}
\label{sec: Universality of the Nullfication Spectral Sequence}

In this section we show that the nullification spectral sequence constructed in Lemma~\ref{lem: Construction of tc-nullification} is, in an appropriate sense, the only one possible. As before $R$ is an $\sphere$-algebra and $\tc$ is a hereditary torsion class on $\pi_*R$-modules.

\begin{definition}
\label{def: Proper nullification spectral sequence}
Let $M$ be an $R$-module. A \emph{proper} $\tc$-nullification spectral sequence for $M$ is a reasonable nullification spectral sequence such that the map $\nu:M \to E^0$ in $\tc$-nullification of $M$. In particular this implies that the nullification candidate is in fact the $\tc$-nullification of $M$.
\end{definition}

\begin{theorem}
\label{thm: Universality for nullification}
Let $R$ be an $\sphere$-algebra and let $\tc$ be a hereditary torsion class on $\pi_* R$-modules. Suppose given an $R$-module $M$ and a proper $\tc$-nullification spectral sequence for $M$. Then this spectral sequence is isomorphic to the nullification spectral sequence constructed in Lemma~\ref{lem: Spectral sequence for nullification}.

\end{theorem}
\begin{proof}
Before we begin we must set some notation. The proper spectral sequence will consist of a $\tc$-nullification map $\nu:M \to F^0$ and maps
\[ \xymatrixcompile{
{\cdots} \ar[r] & {F^2} \ar[rr] && {F^1} \ar[rr] \ar[dl]^k && {F^0} \ar[dl]^k & {M} \ar[l]^\nu \\
   &   &  {\Sigma^{-1} G^1}  \ar@{-->}[ul]^j  &  & {G^0} \ar@{-->}[ul]^j
}\]
where the dashed arrows shift degrees and each triangle $F^{p+1} \to F^p \to \Sigma^{-p} G^p$ is distinguished. The $\Espec^1$-term is then the cochain complex $\pi_* G^0 \xrightarrow{\delta} \pi_* G^1 \xrightarrow{\delta} \cdots$
where $\delta=\pi_*(kj)$. Let $\Bmodule^n = \delta(\pi_*G^n)$ and $\Zmodule^n= \ker(\delta:\pi_*G^n \to \pi_*G^{n+1})$.

Now let $N$ be any $R$-module such that $\pi_* N$ is torsion free injective, we show that there are isomorphisms
\[ \ext^0_R(F^n,N) \xleftarrow{\cong} \hom_{\pi_*R}(\Zmodule^{n}, \pi_* N) \xrightarrow{\cong} \hom_{\pi_*R}(\Bmodule^{n}, \pi_* N) \]
for $n\geq 1$ and
\[ \ext^0_R(F^0,N) \xleftarrow{\cong} \hom_{\pi_*R}(\Zmodule^0,\pi_* N)\]
induced by the obvious maps. First, $N$ is $\tc$-null and therefore $\nu$ induces an isomorphism $\ext^0_R(F^0,N) \cong \ext^0_R(M,N)$. Since the kernel and cokernel of the map $\pi_*M \to \Zmodule^0$ are torsion, we see that $\ext^0_R(M,N) \cong \hom_{\pi_*R}(\Zmodule^0,\pi_* N)$. Next, assuming the isomorphism $\ext^0_R(F^n,N) \cong \hom_{\pi_*R}(\Zmodule^{n}, \pi_* N)$, the short exact sequence $\Zmodule^n \to \pi_*G^n \to \Bmodule^{n+1}$ implies that $\ext^0_R(F^{n+1},N) \cong \hom_{\pi_*R}(\Bmodule^{n+1}, \pi_* N)$. Last, note that for $n>0$ the cohomology groups $\Zmodule^n/\Bmodule^n$ are $\tc$-torsion (since these are also cohomology groups of the colocalization) and therefore $\hom_{\pi_*R}(\Bmodule^{n}, \pi_* N) \cong \hom_{\pi_*R}(\Zmodule^{n}, \pi_* N)$.

Recall that in Lemma~\ref{lem: Spectral sequence for nullification} we constructed a map $M \to N$ such that the module $N$ is $\tc$-null, in what follows we shall use $E^0$ to denote $N$. We are therefore faced with the following diagram which the describes the filtration inducing the nullification spectral sequence of Lemma~\ref{lem: Spectral sequence for nullification}:
\[ \xymatrixcompile{
{\cdots} \ar[r] & {E^2} \ar[rr] && {E^1} \ar[rr] \ar[dl]^{\bar{k}} && {E^0} \ar[dl]^{\bar{k}} & {M} \ar[l]^\psi \\
   &   &  {\Sigma^{-1} N^1}  \ar@{-->}[ul]^{\bar{j}}  &  & {N^0} \ar@{-->}[ul]^{\bar{j}}
}\]
As above the dashed arrows shift degrees and each triangle is distinguished. Denote the spectral sequence induced by this structure by $\bar{\Espec}$. To prove the theorem we will construct in $\Derived(R)$ a morphism of towers $\alpha^p:F^p \to E^p$  which will induce a quasi-isomorphism $b:\Espec^1 \to \bar{\Espec}^1$.


Since $E^0$ is $\tc$-null there is a unique morphism $\alpha^0:F^0 \to E^0$ such that $\alpha^0 \nu = \psi$. The isomorphism $\ext^0_R(F^0,N^0) \cong \hom_{\pi_*R}(\Zmodule^0,\pi_* N^0)$, coupled with the fact that $\pi_*N^0$ is injective, implies that the map $\ext^0_R(G^0,N^0) \to \ext^0_R(F^0,N^0)$ is a surjection. Choose $\beta^0:G^0 \to N^0$ which satisfies $k\beta^0=\bar{k}\alpha^0:F^0 \to N^0$. The axioms of a triangulated category now produce a morphism $\alpha^1:F^1 \to E^1$ making the following diagram commute:
\[ \xymatrixcompile{
{F^1} \ar[r] \ar[d]^{\alpha^1} & {F^0} \ar[r] \ar[d]^{\alpha^0} & {G^0} \ar[r] \ar[d]^{\beta^0} & {\Sigma F^1} \ar[d]^{\Sigma \alpha^1} \\
{E^1} \ar[r] & {E^0} \ar[r] & {N^0} \ar[r] & {\Sigma E^1}
}\]
One readily sees how to continue this process inductively to get a morphism of towers $\alpha$.

The crucial property of the map of complexes $b:\Espec^1 \to \bar{\Espec}^1$ induced by $\beta$ is that $b \pi_*\nu = \pi_* \psi:\pi_*M \to \bar{\Espec}^1$, as is evident from the construction. Since both $\pi_*\nu$ and $\pi_* \psi$ are $\tc$-nullifications of $\pi_*M$ in $\Derived(\pi_*R)$, the morphism $b$ must be a quasi-isomorphism.
\end{proof}

\begin{corollary}
Let $R$ be a commutative $\sphere$-algebra and let $\ideal$ be a finitely generated ideal of $\pi_*R$. Then for any $R$-module $M$ the Greenlees local-cohomology spectral sequence for $M$ is isomorphic to colocalization spectral sequence of Corollary~\ref{cor: The colocalization spectral sequence}.
\end{corollary}
\begin{proof}
Let $\tc$ be the hereditary torsion theory generated by $\ideal$-power torsion modules. It is clear from~\cite{GreenleesMay} that the Greenlees spectral sequence is induced by a proper $\tc$-nullfication spectral sequence for $M$, which can be thought of as a \u{C}ech homology spectral sequence. Since the $\tc$-nullification spectral sequence of Lemma~\ref{lem: Spectral sequence for nullification} is isomorphic to the \u{C}ech homology spectral sequence, the result follows.
\end{proof}

\section{Application: the chains of loops on an elliptic space}
\label{sec: Application: the chains of loops on an elliptic space}

Elliptic spaces, defined by Felix, Halperin and Thomas in~\cite{FelixHalperinThomasEliiptic1}, have a particularly well behaved loop space homology. This enables us to apply Theorem~\ref{thm: Second Theorem} to modules over the $\sphere$-algebra $\chains_*(\kangr X;H\Int)$, where $X$ is an elliptic space. We begin by recalling the definition of an elliptic space from~\cite{FelixHalperinThomasEliiptic1}.

\begin{definition}
\label{def: Elliptic space}
Let $k$ be a sub-ring of $\Q$. A simply connected CW complex $X$ is called \emph{$k$-elliptic} if
\begin{enumerate}
\item $X_k$ has a finite Lusternik-Schnirelmann category, where $X_k$ is the $k$-localization of $X$.
\item Each $H_i(X;k)$ is finitely generated.
\item For $\field=\Q$ or $\field=\Int/p$ for prime $p$, the sequence $\dim_\field H_i(X_k;\field)$ grows at most polynomially.
\end{enumerate}
If $X$ is $\Int$-elliptic we will say that $X$ is \emph{elliptic}.
\end{definition}

\begin{remark}
Note that by~\cite{FelixHalperinThomasEliiptic1}, a $k$-elliptic space $X$ has the $k$-homotopy type of a finite complex. So, for the applications we have in mind we may as well assume that $X$ is a finite complex. Moreover, by~\cite[Theorem A]{FelixHalperinThomasEliiptic1}, if $X$ is $k$-elliptic then $H^*(X;k)$ satisfies Poincar\'{e} duality and if $X$ is elliptic then $X$ has the homotopy type of a finite Poincar\'{e} complex.
\end{remark}

\begin{theorem}
\label{the: Application to the chains of loops of elliptic}
Let $X$ be a elliptic space and let $R$ be the $\sphere$-algebra $\chains_*(\kangr X;H\Int)$. Suppose that $M$ is an $R$-module such that $\pi_*M$ is a finitely generated $\pi_*R$-module. Then there is a strongly convergent colocalization spectral sequence:
\[ \Espec^2_{p,q}= H_{p,q}(\cell_\Int^{\pi_*R} \pi_*M) \ \Rightarrow \ \pi_{p+q} (\cell_{H\Int}^R M)\]
In addition, for $M=R$ we have:
\[ \Espec^2_{p,q}= H_{p,q}(\cell_\Int^{\pi_*R} \pi_*R) \ \Rightarrow \ H^{a-p-q} (\kangr X;\Int)\]
where $a$ is the Poincar\'{e} duality dimension of $X$.
\end{theorem}
\begin{proof}
From~\cite{ShamirGradedColocalization} it follows that $R$ satisfies the conditions of Theorem~\ref{thm: Second Theorem}. Thus, there is a strongly convergent spectral sequence
\[ \Espec^2_{p,q}= H_{p,q}(\cell_{\tc}^{R_*} \pi_*(M)) \ \Rightarrow \ \pi_{p+q}(\cell_\tc^R M)\]
where $\tc$ is the hereditary torsion class of $\ideal$-power torsion modules and $\ideal=H_+(\kangr X)$. Lemma~\ref{lem: Cellularization at R/I is I-power torsion when} shows that $\tc$-colocalization is the same as $\Int$-colocalization. Thus, we can interpret the $\Espec^2$-page of this spectral sequence as
\[ \Espec^2_{p,q}= H_{p,q}(\cell_\Int^{\pi_*R} \pi_*M)\]

The results of~\cite{ShamirGradedColocalization} show that there exists $q_0$ such that $\Espec^2_{p,q}=0$ for all $q>q_0$. Therefore $\cell_\tc M$ is bounded-above. Since $\pi_0 R=\Int$, then Lemma~\ref{lem: Bounded above module spectrum is cellular pi_0-cellular} implies that $\cell_\tc M$ is $H\Int$-cellular. Because the Eilenberg-Mac Lane spectrum $H\Int$ is a $\tc$-cellular $R$-module, then every $\tc$-equivalence is also an $H\Int$-equivalence. We conclude that the $\tc$-colocalization $\cell_\tc M \to M$ is also an $H\Int$-colocalization. Thus we have shown the first spectral sequence.

To complete the proof we need only show that $\cell_{H\Int}M \simeq \Sigma^{-a} \chains^*(\kangr X;\Int)$. Now \cite[10.4 \& 8.2]{DwyerGreenleesIyengar} show that
\[ \cell_{H\Int}R \simeq \Sigma^{-a}H\Int \otimes_{\chains^*(X;\Int)} H\Int\]
The Eilenberg-Moore spectral sequence (see also \cite[7.6]{DwyerGreenleesIyengar}) implies that $H\Int \otimes_{\chains^*(X;H\Int)} H\Int \simeq \chains^*(\kangr X;H\Int)$, which completes the proof.
\end{proof}

It is a simple matter to show that Theorem~\ref{the: Application to the chains of loops of elliptic} applies also when working over a field. However, this is not very interesting as we now explain. Let $k=\Int/p$, let $X$ be a $k$-elliptic space and let $R=\chains_*(X;k)$. Then the results of Felix, Halperin and Thomas from~\cite{FelixHalperinThomasHopfAlgebraElliptic} immediately show there exists $p_0$ such that
\[ H_{p,*}(\cell_k^{\pi_*R} \pi_*R)=\left\{
                                      \begin{array}{ll}
                                        H^{-a-p-*}(\kangr x;k), & p=p_0\hbox{;} \\
                                        0, & \hbox{otherwise.}
                                      \end{array}
                                    \right.\]
A similar result applies when $k=\Q$, this again follows from~\cite{FelixHalperinThomasHopfAlgebraElliptic}.

However, this has the following interesting consequence. Let $X$ be an elliptic space and let $\ring=H_*(\kangr X;H\Int)$. It follows that $\ring\otimes\Q=H_*(\kangr X;\Q)$, let us denote $\ring\otimes\Q$ by $\Qring$. It is not difficult to show there is an equivalence of $\Qring$-complexes
\[\Q \otimes \cell_\Int^\ring (\ring) \simeq \cell_\Q^\Qring \Qring \]
Which proves the following result:
\begin{proposition}
Let $X$ be an elliptic space and let $\ring$ be the graded ring $H_*(\kangr X)$. Then there exists $p_0$ such that $H_{p,q}(\cell_\Int^\ring \ring)\otimes \Q \neq 0$ implies $p=p_0$.
\end{proposition}

\begin{example}
Let $X$ be the Stiefel manifold $V_2(\Real^5)$, i.e. the manifold of all orthonormal 2-frames in $\Real^5$. This is a $7$-dimensional Poincar\'{e} duality manifold. We shall first calculate $H_*(\kangr X)$. There is a fibration $S^3 \to X \to S^4$ and it is well known that $\pi_3(X)=\Int/2$. Examining the Serre spectral sequence for the fibration $\kangr X \to \kangr S^4 \to S^3$ we see there is a long exact sequence
\[ \cdots \to H_{*-2}(\kangr X) \xrightarrow{\theta} H_*(\kangr X) \xrightarrow{\psi} H_*(\kangr S^4) \xrightarrow{\varphi} H_{*-1}(\kangr X) \to \cdots\]
Note that $\psi$ is a map of graded rings, $\theta$ is multiplication by the generator of $H_2(\kangr X)\cong \Int/2$ and $\varphi$ is a map of $H_*(\kangr X)$ modules. From these facts it is a simple exercise to show that $H_*(\kangr X)= \Int[u,v]/2u$ with $|u|=2$ and $|v|=6$.

Let $\ring=H_*(\kangr X)$, let $\ideal=(u,v)$ and let $\tc$ be the class of $\ideal$-power torsion modules. we shall now compute $H_*(\cell_\Int^\ring \ring)$. Since $\ring$ is commutative, $H_{-*}(\cell_\Int^\ring \ring) \cong H^*_\ideal(\ring)$, i.e. the $\ideal$-local cohomology groups of $\ring$ (see~\cite{DwyerGreenlees}). The calculation turns out to be a standard exercise in local-cohomology and we shall spare the reader the details. Denote by $\field_2$ the field $\Int/2$, the local cohomology groups are:
\begin{align*}
H^0_\ideal \ring &=0\\
H^1_\ideal \ring &=\Int[v,1/v]/\Int[v] =\shift^{-6}\Hom_\Int(\Int[v],\Q)\\
H^2_\ideal \ring &=\shift^{-8}\Hom_{\field_2}(\field_2[u,v],\field_2)\\
H^n_\ideal \ring &=0 \quad\quad \text{ for } n>2
\end{align*}
We see that the spectral sequence of Theorem~\ref{the: Application to the chains of loops of elliptic} collapses at the $\Espec^2$-page and indeed yields $H^{*+7} (\kangr X;\Int)$.

When working with coefficients in $\Q$ we have that $\Qring=\ring \otimes \Q=\Q[v]$ and the local cohomology groups are:
\begin{align*}
H^0_\ideal \Qring &=0\\
H^1_\ideal \Qring &=\Q[v,1/v]/\Q[v]=\shift^{-6}\Hom_\Q(\Q[v],\Q) \\
H^2_\ideal \Qring&=0 \quad\quad \text{ for } n\neq1
\end{align*}
On the other hand, set $\Sring=H_*(\kangr X;\field_2)$ then $\Sring=\field_2[u,y]$ where $|y|=3$. Let $\jdeal$ be the maximal ideal of $\Sring$, then the $\jdeal$-local cohomology groups are
\begin{align*}
H^2_\jdeal \Sring &=\shift^{-5} \Hom_{\field_2}(\Sring,\field_2)\\
H^n_\jdeal \Sring &=0 \quad\quad \text{ for } n\neq 2
\end{align*}
\end{example}

\section{Application: the target of the Eilenberg-Moore spectral sequence}
\label{sec: Application: the target of the Eilenberg-Moore spectral sequence}

In this section we will work solely over the field of rational number $\Q$. Let $F \to E \to B$ be a fibration sequence where the spaces $E$ and $B$ are connected and of finite type. The Eilenberg-Moore spectral sequence for this  fibration (with coefficients in $\Q$) has the form:
\[\Espec^2_{p,q}=\tor_{p,q}^{H^*(B;\Q)}(H^*(E;\Q),\Q) \ \Rightarrow \ \pi^{p+q}(\chains^*(E;\Q)\otimes_{\chains^*(B;\Q)} \Q)\]

It is standard that the space $F$ is weakly equivalent to a $\kangr B$-space, thus $\chains^*(F;\Q)$ is a $\chains_*(\kangr B;\Q)$-module. Now suppose that $B$ is a finite CW complex. Then the results of~\cite{DwyerGreenleesIyengar} show there is an equivalence of $\chains_*(\kangr B;\Q)$-modules:
\[ \chains^*(E;\Q)\otimes_{\chains^*(B;\Q)} \Q \simeq \cell_\Q^{\chains_*(\kangr B;\Q)} \chains^*(F;\Q)\]

Thus, when $\chains^*(F;\Q)$ is $\Q$-cellular as a $\chains_*(\kangr B;\Q)$-module, the Eilenberg-Moore spectral sequence converges to the rational cohomology of the fiber $F$. Dwyer's strong convergence result~\cite{DwyerStrong} can be viewed in this light. In this section we shall use a colocalization spectral sequence to describe the relation between $\chains^*(F;\Q)$ and the target of the Eilenberg-Moore spectral sequence: $\cell_\Q^{\chains_*(\kangr B;\Q)} \chains^*(F;\Q)$.

\begin{theorem}
\label{the: Application to rational EMSS}
Let $F \to E \to B$ be a fibration sequence where the spaces $E$ and $B$ are connected. Suppose that:
\begin{enumerate}
\item $B$ is a finite CW complex with a finite fundamental group $G$,
\item the universal cover $\tilde{B}$ of $B$ is rationally elliptic,
\item $H_*(F;\Q)$ is finitely generated as a right $H_*(\kangr B;\Q)$-module.
\end{enumerate}
Denote by $\omega$ the idempotent $1-\frac{1}{|G|}\sum_{g\in G} g$ in $H_0(\kangr B;\Q)$ and let $\Sring$ be the graded algebra $\omega H_*(\kangr B;\Q) \omega$. Then there exists a strongly convergent spectral sequence:
\[\Espec^2_{p,q} \ \Rightarrow \ \pi_{p+q}(\chains^*(E;\Q)\otimes_{\chains^*(B;\Q)} \Q)\]
where $\Espec^2_{p,q}$ is given by:
\begin{align*}
\Espec^2_{0,*} &=H^*(F;\Q)^G \\
\Espec^2_{-1,*}&=\ext_{\Sring^\mathrm{op}}^0(H_*(F;\Q)\omega , H^*(\kangr B;\Q) \omega)/\nn\\
\Espec^2_{-p,*}&=\ext_{\Sring^\mathrm{op}}^p(H_*(F;\Q)\omega , H^*(\kangr B;\Q) \omega) & \text{ for } p>1.
\end{align*}
where $\nn=H^*(F;\Q)/H^*(F;\Q)^G$.
\end{theorem}

A better way, perhaps, to package the same information is the following.
\begin{theorem}
\label{the: Another formulation of Application to rational EMSS}
Let $F \to E \to B$ be a fibration sequence satisfying the conditions of Theorem~\ref{the: Application to rational EMSS}. Denote by $N$ the mapping cone of the map $\chains^*(E;\Q)\otimes_{\chains^*(B;\Q)} \Q \to \chains^*(F;\Q)$. Then there exists a conditionally convergent spectral sequence:
\[\Espec^2_{p,q}= \ext_{\omega H_*(\kangr B;\Q)^\mathrm{op} \omega}^{-p}(H_*(F;\Q)\omega,H^*(\kangr B;\Q) \omega)\ing{q}  \ \Rightarrow \ \pi_{p+q}N\]
where $\omega$ is as described in Theorem~\ref{the: Application to rational EMSS}.
\end{theorem}

\begin{remark}
The reader might wonder why we are using $\Sring^\mathrm{op}$ in the theorem above instead of $\Sring$. Indeed, one can replace $\Sring^\mathrm{op}$ with $\Sring$ and change the resulting right $\Sring$-modules to left modules. The point is that at the beginning of this section we have designated $\chains^*(F;\Q)$ to be a \emph{left} $\chains_*(\kangr B;\Q)$-module, which means that $H_*(F;\Q)$ is a \emph{right} $H_*(\kangr B;\Q)$-module. So the statement of Theorem~\ref{the: Application to rational EMSS} is consistent with our choices.
\end{remark}

We shall only prove Theorem~\ref{the: Application to rational EMSS}, since Theorem~\ref{the: Another formulation of Application to rational EMSS} is essentially the same.

\begin{proof}[Proof of Theorem~\ref{the: Application to rational EMSS}]
Theorem~\ref{thm: Third Theorem} implies there is a strongly convergent spectral sequence
\[ \Espec^2_{p,q}= H_{p,q}(\cell_{\Q}^{H_*(\kangr B;\Q)} H^*(F;\Q)) \ \Rightarrow \ \pi_{p+q}(\cell_\Q^{\chains_*(\kangr B;\Q)} \chains^*(F;\Q))\]
As noted above (see also \cite[Lemma 5.6]{ShamirEMSS}): $\chains^*(E;\Q)\otimes_{\chains^*(B;\Q)} \Q \simeq \cell_\Q^{\chains_*(\kangr B;\Q)} \chains^*(F;\Q)$. Thus, it remains only to describe the $\Espec^2$-page of this spectral sequence. This calculation is done in~\cite{ShamirGradedColocalization}, where it is shown that
\[ \Null_\Q^{H_*(\kangr B;\Q)} H^*(F;\Q) \simeq \Hom_{\Sring^\mathrm{op}}(H_*(F;\Q)\omega, H^*(\kangr B;\Q)\omega)\]
The triangle $\cell_\Q H^*(F;\Q) \to H^*(F;\Q) \to \Null_\Q H^*(F;\Q)$ yields the $\Espec^2$-page described in Theorem~\ref{the: Application to rational EMSS}.
\end{proof}

\begin{remark}
It is easy to see that when $\pi_1(B)=\Int/2$ then
\[ \Sring =\omega H_*(\kangr B;\Q)\omega \cong H_*(\kangr \tilde{B};\Q)^{\Int/2} \quad \text{ and } \quad H_*(\kangr B;\Q)\omega \cong H_*(\kangr \tilde{B};\Q)\]
\end{remark}

\begin{example}
It is amusing to apply Theorem~\ref{the: Application to rational EMSS} to the principal fibration
\[ \Int/2 \to S^{2n} \to \Real P^{2n}\]
The algebra $H_*(\kangr S^{2n};\Q)$ is $\Q[x]$ with $|x|=2n-1$ (note that this is not graded-commutative, it is simply the tensor algebra on a vector space of dimension $1$ and degree $2n-1$). Hence $H_*(\kangr \Real P^{2n};\Q)$ is the semi-direct product $\Q[x]\rtimes \Int/2$. The action of $\Int/2$ on $\pi_{2n}(S^{2n})=\Int$ is the non-trivial action. So, if we denote by $g$ is the generator of $\Int/2$, then $gxg^{-1}=gxg=-x$. Thus:
\begin{align*}
\omega H_*(\kangr \Real P^{2n};\Q) \omega &\cong H_*(\kangr S^{2n};\Q)^{\Int/2} \cong \Q[x^2]\\
\omega H_*(\kangr \Real P^{2n};\Q) &\cong \Q[x]\\
\Q[\Int/2]\omega &\cong \Q
\end{align*}
Note that the isomorphism $\Q[\Int/2]\omega \cong \Q$ is not an isomorphism of $\Q[\Int/2]$-modules, since $\Q[\Int/2]\omega$ has a non-trivial $\Int/2$ action. Nevertheless, this is sufficient for us to compute the $\Espec^2$-page of the spectral sequence of Theorem~\ref{the: Application to rational EMSS}:
\[ \Espec^2_{p,q}=\left\{
                    \begin{array}{ll}
                      \Q, & \hbox{$(p,q)=(0,0)$;} \\
                      \Q, & \hbox{$(p,q)=(-1,2n-1)$;} \\
                      0, & \hbox{otherwise.}
                    \end{array}
                  \right.\]
The spectral sequence then collapses in the $\Espec^2$-page to reveal that
\[ \pi_n(\chains^*(S^{2n};\Q)\otimes_{\chains^*(\Real P^{2n};\Q)}\Q) \cong \left\{
                                                                             \begin{array}{ll}
                                                                               \Q, & \hbox{$n=0$ or $n=-2n$;} \\
                                                                               0, & \hbox{otherwise.}
                                                                             \end{array}
                                                                           \right.\]
On the other hand, $\chains^*(\Real P^{2n};\Q) \simeq \Q$ and hence $\chains^*(S^{2n};\Q)\otimes_{\chains^*(\Real P^{2n};\Q)}\Q \simeq \chains^*(S^{2n};\Q)$. So we have just calculated the rational cohomology of $S^{2n}$.
\end{example}

\bibliographystyle{amsplain}      
\bibliography{bib2007}          

\end{document}

%% file: Definitions.tex
\usepackage{amsmath,amsfonts,amsthm,amssymb}
\usepackage{fullpage}       
\usepackage[all]{xy}
\usepackage{hyperref}

\theoremstyle{plain}        
\newtheorem{theorem}{Theorem}[section]
\newtheorem{lemma}[theorem]{Lemma}
\newtheorem{proposition}[theorem]{Proposition}
\newtheorem{corollary}[theorem]{Corollary}

\newtheorem{specialthm}{Theorem}        

\theoremstyle{definition}   
\newtheorem{definition}[theorem]{Definition}
\newtheorem{example}[theorem]{Example}
\newtheorem{remark}[theorem]{Remark}
\newtheorem{construction}[theorem]{Nullification Construction}  

\theoremstyle{remark}       

\newcommand{\Int}{\mathbb{Z}}   
\newcommand{\Real}{\mathbb{R}}  
\newcommand{\Q}{\mathbb{Q}}     
\newcommand{\sphere}{\mathbb{S}}
\newcommand{\field}{\mathbb{F}} 

\newcommand{\cc}{\mathsf{C}}            
\newcommand{\Derived}{\mathsf{D}}       
\newcommand{\cderived}{\mathsf{C}}      
\newcommand{\Aclass}{\mathfrak{a}}      
\newcommand{\Bclass}{\mathfrak{b}}      
\newcommand{\Cyc}{\mathcal{C}}          
\newcommand{\loc}[1]{{\langle #1 \rangle}}  

\newcommand{\Hom}{\mathrm{Hom}}         
\newcommand{\ext}{\mathrm{Ext}}         
\newcommand{\tor}{\mathrm{Tor}}         
\newcommand{\End}{\mathrm{End}}         
\newcommand{\cone}{\mathrm{Cone}}       
\newcommand{\chains}{\mathrm{C}}        
\newcommand{\cell}{\mathrm{Cell}}       
\newcommand{\Null}{\mathrm{Null}}       
\newcommand{\kangr}{\Omega}             
\newcommand{\B}{\mathrm{B}}             

\newcommand{\ring}{\mathcal{R}}         
\newcommand{\Sring}{\mathcal{S}}        
\newcommand{\Qring}{\mathcal{Q}}        
\newcommand{\ideal}{\mathcal{I}}        
\newcommand{\jdeal}{\mathcal{J}}        
\newcommand{\module}{\mathcal{M}}       
\newcommand{\Bmodule}{\mathcal{B}}      
\newcommand{\Cmodule}{\mathcal{C}}      
\newcommand{\Emodule}{\mathcal{E}}      
\newcommand{\Gmodule}{\mathcal{G}}      
\newcommand{\Nmodule}{\mathcal{N}}      
\newcommand{\Zmodule}{\mathcal{Z}}      
\newcommand{\ee}{\mathcal{E}}           
\newcommand{\nn}{\mathcal{N}}           
\newcommand{\Xcomplex}{\mathcal{X}}     
\newcommand{\Ycomplex}{\mathcal{Y}}     
\newcommand{\Ccomplex}{\mathcal{C}}     
\newcommand{\trans}{\Sigma}         
\newcommand{\shift}{\mathfrak{s}}	
\newcommand{\Sshift}{\mathrm{S}}	
\newcommand{\ing}[1]{\langle #1 \rangle}

\newcommand{\tc}{\mathcal{T}}   
\newcommand{\ff}{\mathcal{F}}   

\newcommand{\aideal}{\mathfrak{a}}

\newcommand{\Espec}{\mathbf{E}}         
\newcommand{\Dspec}{\mathbf{D}}         

\newcommand{\Endring}{\mathbb{E}}       

\newcommand{\discard}[1]{}
